\documentclass[12pt, reqno]{amsart}

\usepackage{amsmath, amsthm, amssymb, enumerate, amscd}
\usepackage[mathscr]{eucal}
\usepackage{pgf,tikz}
\usepackage{tikz, tikz-cd}
\usetikzlibrary{matrix,decorations.pathmorphing,arrows}

\theoremstyle{plain}
\newtheorem{theorem}{Theorem}[section]
\newtheorem{proposition}[theorem]{Proposition}

\theoremstyle{definition}
\newtheorem{definition}[theorem]{Definition}

\newtheorem{remark}[theorem]{Remark}

\newcommand{\uu}{{\mathbf{u}}}
\newcommand{\vv}{{\mathbf{v}}}

\begin{document}

\title[Formula for biggest minimal distance] {The formula for the largest minimal distance of binary LCD $[n,2]$ codes}

\author{Seth Gannon}
\address{Department of Mathematics\\ 
University of Louisville\\
Louisville, KY 40292, USA}
\email{dalton.gannon@louisville.edu}
\thanks{$\dag$ the corresponding author}
\author{Hamid Kulosman$^\dag$}
\address{Department of Mathematics\\ 
University of Louisville\\
Louisville, KY 40292, USA}
\email{hamid.kulosman@louisville.edu}
\subjclass[2010]{Primary 94B05}

\keywords{Complementary dual code; LCD code; Minimal distance of a linear code; $[n,2]$ binary code}

\date{}

\begin{abstract} 
In the 2017 paper by Dougherty, Kim, Ozkaya, Sok, and Sol\'e about the linear programming bound for LCD codes the notion $\mathrm{LCD}[n,k]$ was defined for binary LCD $[n,k]$-codes. We find the formula for $\mathrm{LCD}[n,2]$.
\end{abstract}

\maketitle

\section{Introduction}

A {\it linear code with complementary dual} (or an {\it LCD code} for short) is a linear code $C$ whose dual satisfies $C\cap C^\perp=\{0\}$. It was defined in \cite{m}, where a necessary and sufficient condition for a linear code over a field to be an LCD code was given in terms of the generator matrix. The LCD codes are being of considerable interest in the last few years since there are several newly discovered applications of them, including the applications in Quantum Coding Theory. One of important recent papers about LCD codes is the paper \cite{dkoss} which can serve as a foundational paper for a systematic investigation of LCD codes. In that paper, among other things, the authors introduced the notion $\mathrm{LCD}[n,k]$ for binary LCD $[n,k]$ codes and gave the values of $\mathrm{LCD}[n,2]$ for $n=3, 4, 5, 6,$ and $7$. In this paper we find a general formula for $\mathrm{LCD}[n,2]$.

\smallskip
The reader can consult \cite{hp} for all the notions that we use but do not define in this paper.  We will often be using the following theorem from \cite{m}:

\medskip\noindent
{\bf Massey's Theorem.} {\it If $G$ is a generator matrix for the $[n,k]$ linear code $C$ over a field $F$, then $C$ is an LCD code if and only if the $k\times k$ matrix $GG^T$ is nonsingular.}

\bigskip
\section{Results}

\begin{definition}[\cite{dkoss}]\label{LCD_def}
The number $\mathrm{LCD}[n,k]$ is defined in the following way:
\[\mathrm{LCD}[n,k]=\max\{d\;|\text{ there exists a binary } [n,k,d] \text{ LCD code\}}.\]
\end{definition}

The following values for $\mathrm{LCD}[n,2]$ were given in \cite[Theorem 3.4]{dkoss}: 
\begin{align*}
\mathrm{LCD}[3,2]=2,\\ 
\mathrm{LCD}[4,2]=2,\\ 
\mathrm{LCD}[5,2]=2,\\ 
\mathrm{LCD}[6,2]=3,\\ 
\mathrm{LCD}[7,2]=4.
\end{align*}

\medskip
In the next few propositions and a theorem we find a general formula for $\mathrm{LCD}[n,2]$. The word ``code" from now on means "binary linear code". Whenever we give a generator matrix for an $[n,2]$ code in standard form $G=[I_2\,|\,A]$, we will denote the word in the first row of $G$ by $\uu$ and the word in the second row of $G$ by $\vv$. Also we will call the submatrix $A$ of $G$ the {\it extension part} of $G$ and the digits of $\uu$ and $\vv$ that are in $A$ the {\it extension digits} of $\uu$ and $\vv$.

\begin{proposition}\label{lower_bound}
For any integer $r\ge 0$ we have:
\begin{align*}
\mathrm{LCD}[6r+3,2] &\ge 4r+2,\\
\mathrm{LCD}[6r+4,2] &\ge 4r+2,\\
\mathrm{LCD}[6r+5,2] &\ge 4r+2,\\
\mathrm{LCD}[6r+6,2] &\ge 4r+3,\\
\mathrm{LCD}[6r+7,2] &\ge 4r+4,\\
\mathrm{LCD}[6r+8,2] &\ge 4r+5.
\end{align*}
\end{proposition}
\begin{proof}
For $r\ge 0$ and $s\in\{3,4,5\}$ let $C$ be the code with generator matrix in standard form

\medskip
\centerline{
\begin{tikzpicture}
\draw (1,0) node
{
$G=\left[ \begin{array}{cc|c|cc|ccc}
1 & 0 & \cdots & 1 & 0 & 1 & \cdots & 1\\
0 & 1 & \cdots & 0 & 1 & 1 &  \cdots & 1
\end{array}
\right]\!\!.
$
};
\draw (.45,-.55) node {$\underbrace{\phantom{MMMMMMa}}$};
\draw (.5, -1.2) node{$2r+1$ $I_2$'s};
\draw (.49, -1.7) node {$= 4r+2$ digits};
\draw (3,-.55) node {$\underbrace{\phantom{MMMM}}$};
\draw (3.1, -1.2) node {$2r+s-2$};
\draw (3.15, -1.7) node {ones};
\end{tikzpicture}
}

\noindent
Then $wt(\uu)=wt(\vv)=4r+s-1$ and $wt(\uu+\vv)=4r+2$, so that $d=4r+2$. Using the block multiplication of matrices we conclude from Massey's theorem that $C$ is LCD. Hence the first three inequalities hold.

For $r\ge 0$ and $s\in\{6,7,8\}$ let $C$ be the code with generator matrix in standard form

\medskip
\centerline{
\begin{tikzpicture}
\draw (1,0) node
{
$G=\left[ \begin{array}{cc|c|cc|ccc}
1 & 0 & \cdots & 1 & 0 & 1 & \cdots & 1\\
0 & 1 & \cdots & 0 & 1 & 1 &  \cdots & 1
\end{array}
\right]\!\!.
$
};
\draw (.45,-.55) node {$\underbrace{\phantom{MMMMMMa}}$};
\draw (.5, -1.2) node{$2r+3$ $I_2$'s};
\draw (.49, -1.7) node {$= 4r+6$ digits};
\draw (3,-.55) node {$\underbrace{\phantom{MMMM}}$};
\draw (3.1, -1.2) node {$2r+s-6$};
\draw (3.15, -1.7) node {ones};
\end{tikzpicture}
}

\noindent
Then $wt(\uu)=wt(\vv)=4r+s-3$ and $wt(\uu+\vv)=4r+6$, so that $d=4r+s-3$. Using the block multiplication of matrices we conclude from Massey's theorem that $C$ is LCD.
Hence the last three inequalities hold.
\end{proof}

\begin{proposition}\label{upper_bound_1}
For any integer $r\ge 0$ we have:
\begin{align*}
\mathrm{LCD}[6r+3,2] &< 4r+3,\\
\mathrm{LCD}[6r+4,2] &< 4r+3,\\
\mathrm{LCD}[6r+7,2] &< 4r+5,\\
\mathrm{LCD}[6r+8,2] &< 4r+6.
\end{align*}
\end{proposition}
\begin{proof}
We prove the first inequality. If $r=0$ it is clearly true. Assume $r\ge 1$. Suppose to the contrary. Let $C$ be an LCD $[6r+3,2]$ code with $d\ge 4r+3$. Up to permutation equivalence we may assume that the generator matrix $G$ of $C$ is in standard form. Then $\uu$ and $\vv$  have at least $4r+2$ extension digits one. Up to permutation equivalence we may assume that the first $4r+2$ extension digits of $\uu$ are ones and that the first $2r+3$ extension digits of $\vv$ are ones. So we have

\medskip
\centerline{
\begin{tikzpicture}
\draw (1,0) node
{
$G=\left[ \begin{array}{cc|ccc|ccc|c}
1 & 0 & 1 & \cdots & 1 & 1 & \cdots & 1 & \phantom{aa}\cdots\phantom{aa}\\
0 & 1 & 1 & \cdots & 1 & & \cdots & & \cdots
\end{array}
\right]\!\!.
$
};
\draw (.12,-.55) node {$\underbrace{\phantom{MMMM}}$};
\draw (.12, -1.2) node{$2r+3$};
\draw (2.1,-.55) node {$\underbrace{\phantom{MMMM}}$};
\draw (2.1, -1.2) node {$2r-1$};
\draw (4.1,-.55) node {$\underbrace{\phantom{MMMM}}$};
\draw (4.15, -1.2) node {$2r-1$};
\end{tikzpicture}
}

\noindent
Now $\uu+\vv$ can have at most $2+(2r-1)+(2r-1)=4r$ ones, contradicting the assumption $d\ge 4r+3$. The first inequality is proved.

The proofs of the remaining three inequalities go along the same lines.
\end{proof}

\begin{proposition}\label{upper_bound_2}
For any integer $r\ge 0$ we have
\[\mathrm{LCD}[6r+6,2] < 4r+4.\]
\end{proposition}
\begin{proof}
 Suppose to the contrary. Let $C$ be an LCD $[6r+6,2]$ code with $d\ge 4r+4$. Up to permutation equivalence we may assume that the generator matrix $G$ of $C$ is in standard form. Then $\uu$ and $\vv$  have at least $4r+3$ extension digits one. Up to permutation equivalence we may assume that the first $4r+3$ extension digits of $\uu$ are ones and that the first $2r+2$ extension digits of $\vv$ are ones. So we have

\medskip
\centerline{
\begin{tikzpicture}
\draw (1,0) node
{
$G=\left[ \begin{array}{cc|ccc|ccc|c}
1 & 0 & 1 & \cdots & 1 & 1 & \cdots & 1 & \phantom{aa}\cdots \phantom{aa}\\
0 & 1 & 1 & \cdots & 1 & & \cdots & & \cdots
\end{array}
\right]\!\!.
$
};
\draw (.13,-.55) node {$\underbrace{\phantom{MMMM}}$};
\draw (.13, -1.2) node{$2r+2$};
\draw (2.1,-.55) node {$\underbrace{\phantom{MMMM}}$};
\draw (2.1, -1.2) node {$2r+1$};
\draw (2.1, -1.7) node {block A};
\draw (4.1,-.55) node {$\underbrace{\phantom{MMMM}}$};
\draw (4.15, -1.2) node {$2r+1$};
\draw (4.15, -1.7) node {block B};
\end{tikzpicture}
}

\noindent
Note the following two things:

(1) $\uu+\vv$ has at least $4r+4$ ones, hence all digits of $v$ in the blocks A and B are opposite to the digits of $\uu$;

(2) at least $2r+1$ digits of $\vv$ in the blocks A and B  are ones.

\noindent
These two things force $G$ to have the following form:

\medskip
\centerline{
\begin{tikzpicture}
\draw (1,0) node
{
$G=\left[ \begin{array}{cc|ccc|ccc|ccc}
1 & 0 & 1 & \cdots & 1 & 1 & \cdots & 1 & 0 & \cdots & 0\\
0 & 1 & 1 & \cdots & 1 & 0 & \cdots & 0 & 1 & \cdots & 1
\end{array}
\right]\!\!.
$
};
\draw (0,-.55) node {$\underbrace{\phantom{MMMM}}$};
\draw (0, -1.2) node{$2r+2$};
\draw (1.98,-.55) node {$\underbrace{\phantom{MMMM}}$};
\draw (1.98, -1.2) node {$2r+1$};
\draw (4.1,-.55) node {$\underbrace{\phantom{MMMM}}$};
\draw (4.15, -1.2) node {$2r+1$};
\end{tikzpicture}
}

Hence up to permutation equivalence 

\medskip
\centerline{
\begin{tikzpicture}
\draw (1,0) node
{
$G=\left[ \begin{array}{cc|c|cc|ccc}
1 & 0 & \cdots & 1 & 0 & 1 & \cdots & 1\\
0 & 1 & \cdots & 0 & 1 & 1 &  \cdots & 1
\end{array}
\right]\!\!.
$
};
\draw (.45,-.55) node {$\underbrace{\phantom{MMMMMMa}}$};
\draw (.5, -1.2) node{$2r+2$ $I_2$'s};
\draw (.49, -1.7) node {$= 4r+4$ digits};
\draw (3,-.55) node {$\underbrace{\phantom{MMMM}}$};
\draw (3.1, -1.2) node {$2r+2$};
\draw (3.15, -1.7) node {ones};
\end{tikzpicture}
}

\noindent
Using the block multiplication of matrices we conclude from Massey's theorem that $C$ is not LCD. We got a contradiction, the inequality is proved.
\end{proof}

\begin{proposition}\label{upper_bound_3}
For any integer $r\ge 0$ we have
\[\mathrm{LCD}[6r+5,2] < 4r+3.\]
\end{proposition}
\begin{proof}
 Suppose to the contrary. Let $C$ be an LCD $[6r+5,2]$ code with $d\ge 4r+3$. Up to permutation equivalence we may assume that the generator matrix $G$ of $C$ is in standard form. Then $\uu$ and $\vv$  have at least $4r+2$ extension digits one. Up to permutation equivalence we may assume that the first $4r+2$ extension digits of $\uu$ are ones and that the first $2r+1$ extension digits of $\vv$ are ones. So we have 

\medskip
\centerline{
\begin{tikzpicture}
\draw (1,0) node
{
$G=\left[ \begin{array}{cc|ccc|ccc|c}
1 & 0 & 1 & \cdots & 1 & 1 & \cdots & 1 & \phantom{aa}\cdots \phantom{aa}\\
0 & 1 & 1 & \cdots & 1 & & \cdots & & \cdots
\end{array}
\right]\!\!.
$
};
\draw (.13,-.55) node {$\underbrace{\phantom{MMMM}}$};
\draw (.13, -1.2) node{$2r+1$};
\draw (2.1,-.55) node {$\underbrace{\phantom{MMMM}}$};
\draw (2.1, -1.2) node {$2r+1$};
\draw (2.1, -1.7) node {block A};
\draw (4.1,-.55) node {$\underbrace{\phantom{MMMM}}$};
\draw (4.15, -1.2) node {$2r+1$};
\draw (4.15, -1.7) node {block B};
\end{tikzpicture}
}

\noindent
Note the following two things:

(1) $\uu+\vv$ has at least $4r+3$ ones, hence all digits of $v$ in the blocks A and B, except at most one, are opposite to the digits of $\uu$;

(2) at least $2r+1$ digits of $\vv$ in the blocks A and B  are ones.

\noindent
Hence, up to permutation equivalence, $G$ has the following form:

\medskip
\centerline{
\begin{tikzpicture}
\draw (1,0) node
{
$G=\left[ \begin{array}{cc|ccc|cccc|ccccc}
1 & 0 & 1 & \cdots & 1 & 1 & \cdots & 1 & 1 & b_1 & b_2 & \cdots & b_{2r} & c \\
0 & 1 & 1 & \cdots & 1 & 0 & \cdots & 0 & a & \overline{b_1} & \overline{b_2} & \cdots & \overline{b_{2r}} & d
\end{array}
\right]\!\!,
$
};
\draw (-1.18,-.55) node {$\underbrace{\phantom{MMMM}}$};
\draw (-1.18, -1.2) node{$2r+1$};
\draw (.91,-.55) node {$\underbrace{\phantom{MMMm}}$};
\draw (.91, -1.2) node {$2r$};
\end{tikzpicture}
}
\noindent
where the overline denotes the opposite digit. If $r=0$, we have 
$G=\left[\begin{array}{cc|ccc}
1 & 0 & 1 & 1 & c\\
0 & 1 & 1 & a & d
\end{array}
\right]\!\!.$
By (2) at least one of $a,d$ is $1$. If $a=1$, then (1) implies $d=\overline{c}$, so that 
$G=\left[\begin{array}{cc|ccc}
1 & 0 & 1 & 1 & c\\
0 & 1 & 1 & 1 & \overline{c}
\end{array}
\right]\!\!.$ But then $GG^T$ is equal to either $\begin{bmatrix} 1 & 0\\ 0 & 0\end{bmatrix}$, or to $\begin{bmatrix} 0 & 0\\ 0 & 1\end{bmatrix}$\!, so that, by Massey's theorem, $C$ is not LCD, a contradiction.  If $a=0$, then (2) implies $d=1$, so that 
$G=\left[\begin{array}{cc|ccc}
1 & 0 & 1 & 1 & c\\
0 & 1 & 1 & 0 & 1
\end{array}
\right]\!\!.$ But then $GG^T$ is equal to either $\begin{bmatrix} 1 & 1\\ 1 & 1\end{bmatrix}$, or to $\begin{bmatrix} 0 & 0\\ 0 & 1\end{bmatrix}$\!, so that, by Massey's theorem, $C$ is not LCD, a contradiction.  Assume now that $r\ge 1$.
Because of (1) we have either $a=0$, or $d=\overline{c}$. Because of (2), among the digits $a, \overline{b_1}, \overline{b_2}, \dots, \overline{b_{2r}},d$, the word $\vv$ has at least $2r+1$ ones. Hence among the digits $\overline{b_1}, \overline{b_2}, \dots, \overline{b_{2r}}$, the word $\vv$ has at least $2r-1$ ones. Hence, up to permutation equivalence, $G$ has the following form:

\medskip
\centerline{
\begin{tikzpicture}
\draw (1,0) node
{
$G=\left[ \begin{array}{cc|ccc|cccc|ccccc}
1 & 0 & 1 & \cdots & 1 & 1 & \cdots & 1 & 1 & 0 & \cdots & 0  & b_{2r} & c \\
0 & 1 & 1 & \cdots & 1 & 0 & \cdots & 0 & a & 1 & \cdots & 1 & \overline{b_{2r}} & d
\end{array}
\right]\!\!.
$
};
\draw (-1.18,-.55) node {$\underbrace{\phantom{MMMM}}$};
\draw (-1.18, -1.2) node{$2r+1$};
\draw (1,-.55) node {$\underbrace{\phantom{MMMM}}$};
\draw (1, -1.2) node {$2r$};
\draw (3.57,-.55) node {$\underbrace{\phantom{MMMM}}$};
\draw (3.57, -1.2) node {$2r-1$};
\end{tikzpicture}
}

\noindent
Here because of (1) either $a=0$, or $d=\overline{c}$, and, because of (2), at least two of the digits $a, \overline{b_{2r}}, d$ are ones. So we have the following options:

(i) $a=0$, $\overline{b_{2r}}=d=1$ (so that $b_{2r}=0)$;

(ii) $a=1$, $d=\overline{c}$, and $d=1$ (so that $c=0$);

(iii) $a=1$, $d=\overline{c}$, and $\overline{b_{2r}}=1$.

\noindent
This implies that for the matrix $G$, we, respectively, have the following options (i), (ii), and (iii):

\medskip
\centerline{
\begin{tikzpicture}
\draw (1,0) node
{
$G=\left[ \begin{array}{cc|ccc|cccc|ccccc}
1 & 0 & 1 & \cdots & 1 & 1 & \cdots & 1 & 1 & 0 & \cdots & 0  & 0 & c \\
0 & 1 & 1 & \cdots & 1 & 0 & \cdots & 0 & 0 & 1 & \cdots & 1 & 1 & 1
\end{array}
\right]\!\!,
$
};
\draw (-.9,-.55) node {$\underbrace{\phantom{MMMM}}$};
\draw (-.9, -1.2) node{$2r+1$};
\draw (1.1,-.55) node {$\underbrace{\phantom{MMMM}}$};
\draw (1.1, -1.2) node {$2r$};
\draw (3.75,-.55) node {$\underbrace{\phantom{MMMM}}$};
\draw (3.75, -1.2) node {$2r-1$};
\end{tikzpicture}
}

\medskip
\centerline{
\begin{tikzpicture}
\draw (1,0) node
{
$G=\left[ \begin{array}{cc|ccc|cccc|ccccc}
1 & 0 & 1 & \cdots & 1 & 1 & \cdots & 1 & 1 & 0 & \cdots & 0  & b_{2r} & 0 \\
0 & 1 & 1 & \cdots & 1 & 0 & \cdots & 0 & 1 & 1 & \cdots & 1 & \overline{b_{2r}} & 1
\end{array}
\right]\!\!,
$
};
\draw (-.98,-.55) node {$\underbrace{\phantom{MMMM}}$};
\draw (-.98, -1.2) node{$2r+1$};
\draw (1,-.55) node {$\underbrace{\phantom{MMMM}}$};
\draw (1, -1.2) node {$2r$};
\draw (3.62,-.55) node {$\underbrace{\phantom{MMMM}}$};
\draw (3.62, -1.2) node {$2r-1$};
\end{tikzpicture}
}

\medskip
\centerline{
\begin{tikzpicture}
\draw (1,0) node
{
$G=\left[ \begin{array}{cc|ccc|cccc|ccccc}
1 & 0 & 1 & \cdots & 1 & 1 & \cdots & 1 & 1 & 0 & \cdots & 0  & 0 & c \\
0 & 1 & 1 & \cdots & 1 & 0 & \cdots & 0 & 1 & 1 & \cdots & 1 & 1 & \overline{c}
\end{array}
\right]\!\!.
$
};
\draw (-.9,-.55) node {$\underbrace{\phantom{MMMM}}$};
\draw (-.9, -1.2) node{$2r+1$};
\draw (1.2,-.55) node {$\underbrace{\phantom{MMMM}}$};
\draw (1.2, -1.2) node {$2r$};
\draw (3.72,-.55) node {$\underbrace{\phantom{MMMM}}$};
\draw (3.72, -1.2) node {$2r-1$};
\end{tikzpicture}
}

\noindent
\underbar{Option (i):} In this option, up to permutation equivalence, the matrix $G$ has the form

\medskip
\centerline{
\begin{tikzpicture}
\draw (1,0) node
{
$G=\left[ \begin{array}{cc|ccc|cc|c|cc|cc}
1 & 0 & 1 & \cdots & 1 & 1 & 0 & \cdots & 1 & 0 & 1 & c\\
0 & 1 & 1 & \cdots & 1 & 0 & 1 & \cdots & 0 & 1 & 0 & 1  
\end{array}
\right]\!\!.
$
};
\draw (-.16,-.55) node {$\underbrace{\phantom{MMMm}}$};
\draw (-.16, -1.2) node{$2r+1$};
\draw (2.5,-.55) node {$\underbrace{\phantom{MMMMMMi}}$};
\draw (2.5, -1.2) node {$2r\, I_2$'s};
\end{tikzpicture}
}

\medskip\noindent
Note that $\begin{bmatrix} 1 & c\\ 0 & 1\end{bmatrix}\,\begin{bmatrix} 1 & c\\ 0 & 1\end{bmatrix}^T$ is equal to either $\begin{bmatrix} 1 & 0\\ 0 & 1\end{bmatrix}$\!, or to $\begin{bmatrix} 0 & 1\\ 1 & 1\end{bmatrix}$. Hence, using the block multiplication of matrices, we conclude that $GG^T$ is equal to either $\begin{bmatrix} 1 & 1\\ 1 & 1\end{bmatrix}$\!, or to $\begin{bmatrix} 0 & 1\\ 0 & 1\end{bmatrix}$\!. So by Massey's theorem the code $C$ is not LCD and we have a contradiction.

\medskip\noindent
\underbar{Option (ii):} In this option, up to permutation equivalence, the matrix $G$ has the form

\medskip
\centerline{
\begin{tikzpicture}
\draw (1,0) node
{
$G=\left[ \begin{array}{cc|ccc|cc|c|cc|cc}
1 & 0 & 1 & \cdots & 1 & 1 & 0 & \cdots & 1 & 0 & 1 & b_{2r}\\
0 & 1 & 1 & \cdots & 1 & 0 & 1 & \cdots & 0 & 1 & 1 & \overline{b_{2r}}  
\end{array}
\right]\!\!.
$
};
\draw (-.26,-.55) node {$\underbrace{\phantom{MMMm}}$};
\draw (-.26, -1.2) node{$2r+1$};
\draw (2.3,-.55) node {$\underbrace{\phantom{MMMMMMi}}$};
\draw (2.3, -1.2) node {$2r\, I_2$'s};
\end{tikzpicture}
}

\medskip\noindent
Reasoning similarly as in option (i) we conclude that $GG^T$ is equal to either $\begin{bmatrix} 0 & 0\\ 0 & 1\end{bmatrix}$\!, or to $\begin{bmatrix} 1 & 0\\ 0 & 0\end{bmatrix}$\!. So by Massey's theorem the code $C$ is not LCD and we have a contradiction.

\medskip\noindent
\underbar{Option (iii):} This option is analyzed in the same way as the option (ii).

\medskip\noindent
Since we got a contradiction with the assumption that $C$ is LCD, the inequality is proved.
\end{proof}

\medskip
We now state our main theorem.

\begin{theorem}
For any integers $r\ge 0$ and $s\in\{3,4,5,6,7,8\}$ we have:
\begin{align*}
\mathrm{LCD}[6r+3,2] &= 4r+2,\\
\mathrm{LCD}[6r+4,2] &= 4r+2,\\
\mathrm{LCD}[6r+5,2] &= 4r+2,\\
\mathrm{LCD}[6r+6,2] &= 4r+3,\\
\mathrm{LCD}[6r+7,2] &= 4r+4,\\
\mathrm{LCD}[6r+8,2] &= 4r+5.
\end{align*}
In other words:
\[\mathrm{LCD}[6r+s,2]=4r+\left\lfloor\frac{s}{6}\right\rfloor\!(1+s\negmedspace\negmedspace\negmedspace\mod 6)+2.\]
\end{theorem}
\begin{proof}
The theorem follows from the propositions \ref{lower_bound}, \ref{upper_bound_1}, \ref{upper_bound_2}, and \ref{upper_bound_3}.
\end{proof}

\begin{remark}
Note that the last equality of the previous theorem holds also when $r=-1$; it gives $\mathrm{LCD}[2,2]=1$.
\end{remark}

\end{document}